\documentclass[11pt,a4paper]{article}
\usepackage[OT1]{fontenc}
\usepackage[english]{babel}
\usepackage{amsfonts}
\usepackage{amsthm}
\def\h{ {\cal H} }

\def\l{ {\cal L} }

\def\b{ {\cal B} }

\def\t{ {\cal T} }
\def\s{ {\cal S} }

\def\p{ {\cal P} }

\def\k{ {\cal K} }
\def\f{ {\cal F} }

\def\sinc{ \hbox{sinc} }

\newtheorem{teo}{Theorem}[section]
\newtheorem{prop}[teo]{Proposition}
\newtheorem{lem}[teo]{Lemma}
\newtheorem{coro}[teo]{Corollary}

\theoremstyle{definition}
\newtheorem{rem}[teo]{Remark}
\newtheorem{ejem}[teo]{Example}
\newtheorem{ejems}[teo]{Examples}

\title{Schmidt decomposable products of projections}
\author{Esteban Andruchow and Gustavo Corach}

\begin{document}

\maketitle 

\begin{abstract}
We characterize operators $T=PQ$ ($P,Q$ orthogonal projections in a Hilbert space $\h$) which have a singular value decomposition. A spatial characterizations is given: this condition occurs if and only if there exist orthonormal bases $\{\psi_n\}$ of $R(P)$ and $\{\xi_n\}$ of $R(Q)$ such that 
$\langle\xi_n,\psi_m\rangle=0$ if $n\ne m$. Also it is shown that this is equivalent  to $A=P-Q$ being diagonalizable. Several examples are studied, relating Toeplitz, Hankel and Wiener-Hopf operators to this condition. We also examine the relationship with the differential geometry of the Grassmann manifold of underlying the Hilbert space: if $T=PQ$ has a singular value decomposition, then the generic parts of $P$ and $Q$ are joined by a minimal geodesic with diagonalizable exponent.
\end{abstract}

\bigskip



\section{Introduction}

Let $\h$ be a Hilbert space, $\b(\h)$ the space of bounded linear operators, $\p(\h)\subset \b(\h)$ the set of orthogonal projections. In what follows $R(T)$ denotes the the range of $T\in\b(\h)$ and $N(T)$ its nullspace. Given a closed subspace $\s\subset \h$, the orthogonal projection  onto $\s$ is denoted by $P_\s$. In this paper we study  part of the set $\p\cdot\p=\{PQ: P,Q\in\p(\h)\}$, namely, the subset of all $T=PQ$ such that $T^*T=PQP$ is diagonalizable.   
Operators in $\p\cdot\p$ are special cases of generalized Toeplitz operators as well as of Wiener-Hopf operators. As we shall see in a section of examples,  they give rise to classical Toeplitz and Wiener-Hopf operators. Therefore this paper can be regarded as the study of operators in these classes, having a diagonal structure.

Also this paper is a kind of sequel to  \cite{cm}, \cite{incertidumbre} and \cite{jot}, the first concerned with the whole set $\p\cdot\p$, the other two with $\p\cdot\p\cap \k(\h)$, where $\k(\h)$ denotes the ideal of compact operators acting in $\h$. Compact operators $T$ satisfy that $T^*T$ is diagonalizable. 

We shall say that $ T$ is S-decomposable if it has a  singular value (or Schmidt) decomposition \cite{schmidt},
\begin{equation}\label{S}
T=\sum_{n\ge 1} s_n\langle\ \ \ , \xi_n\rangle\psi_n=\sum_{n\ge 1}s_n \psi_n\otimes\xi_n,
\end{equation}
where $\{\xi_n: n\ge 1\}$ and $\{\psi_n: n\ge 1\}$ are orthonormal systems, and $s_n>0$. In this case, $\{\psi_n\}$, $\{xi_n\}$ are orthonormal bases of $\overline{R(T)}$, $N(T)^\perp$, respectively and $T\xi_n=s_n\psi_n$, $T^*\psi_n=s_n\xi_n$, $T^*T\xi_n=s_n^2\xi_n$, $TT^*\psi_n=s_n^2\psi_n$ for all $n\ge 1$.

Clearly, $T$ is S-decomposable if and only if $T^*T$ (equivalently $TT^*$) is diagonalizable, if and only if $T^*$ is S-decomposable. Also it is clear that if $U, V$ are unitary operators, $T$ is S-decomposable if and only if $UTV$ is S-decomposable.

This paper is devoted to the study of the operators $T\in\p\cdot\p$ which are S-decomposable.

Let us describe the contents of the paper. In Section 2 we prove that $T=PQ$ is S-decomposable if and only if there exist orthonormal bases $\{\xi_n\}$  of $R(P)$ and $\{\psi_n\}$ of $R(Q)$ such that $\langle\xi_n,\psi_m\rangle\ne 0$ if $n\ne m$. We also prove that $T=PQ$ is S-decomposable if and only if $A=P-Q$ is diagonalizable. This result is based on a Theorem by Chandler Davis (\cite{davis}, Theorem 6.1), which characterizes operators which are the difference of two projections. A recent treatment of these operators can be found in \cite{pemenoscu}. The S-decomposability of $PQ$ is equivalent to that of $P(1-Q)$, $(1-P)Q$ and $(1-P)(1-Q)$. As a corollary we prove that $P-Q$ is diagonalizable if and only if $P+Q$ is, the eigenvalues $\pm\lambda_n$ of $P-Q$ which are different
 from $0,1$ correspond with the eigenvalues $1\pm(1-\lambda_n)^2$, with the same multiplicity.
Section 3 contains several interesting classes of examples of S-decomposable operators in $\p\cdot\p$. If $\h=L^2(\mathbb{R}^n)$ and $I, J\subset \mathbb{R}^n$ are Lebesgue measurable set with finite positive measure, define $P_If=\chi_If$ and $Q_Jf=(P_J\hat{f} )\check{}$, for $f\in\h$. Here $\chi_A$ denotes the characteristic function of $A\subset \mathbb{R}^n$ and $\hat{}$, $\check{}$ denote the Fourier-Plancherel transform and its inverse. The product $P_IQ_J$ is a proper Wiener-Hopf operator, is also known as a concentration operator, and its study is related to mathematical formulations of the Heisenberg uncertainty principle. The reader is referred to \cite{sp}, \cite{lenard}, \cite{donoho}, \cite{folland}  for results concerning these products. Under the conditions described above, $P_IQ_J$ is a Hilbert-Schmidt operator, thus S-decomposable. This implies that also $P_{I'}Q_{J'}$ is S-decomposable (but non compact) when $I'$ or $J'$ have co-finite measure. It should be mentioned that the spectral  description of the $P_IQ_J$ is no easy task (see \cite{sp} for the case when $I$, $J$ are intervals in $\mathbb{R}$). Another interesting family of examples  is obtained if $\h=L^2(\mathbb{T})$ and $\h$ is decomposed as $\h=\h_+\oplus\h_-$, where $\h_+=H^2(\mathbb{T})$.  If $\varphi$, $\psi$ are continuous functions with modulus one, put $P=1-P_{\varphi\h_+}$ and $Q=P_{\psi\h_+}$. Then $PQ$ is a unitary operator times a Hankel operator with continuous symbol, and therefore a compact operator by a Theorem  by Hartman \cite{hartman}. Then $(1-P)Q$ is a unitary operator times a Toeplitz operator, and a non compact S-decomposable operator. On the other hand, using a result by Howland (\cite{howland}, Theorem 9.2), one can find convenient non-continuous  $\varphi$, $\psi$ such that $PQ$ is not S-decomposable.

In Section 4 we prove that, for two closed subspaces  $\s,\t$ of $\h$, the operator $T=P_\s P_\t$ is S-decomposable if and only if there exist  isometries $X,Y:\ell^2\to\h$ with $R(X)=\s$, $R(Y)=\t$ such that $X^*Y\in\b(\ell^2)$ is a diagonal matrix.

In Section 5 we characterize S-decomposability in terms of what we call Davis' symmetry $V$: given two projections $P,Q$, the decomposition $\h=N(P+Q-1)\oplus N(P+Q-1)^\perp$ reduces simultaneously $P$ and $Q$. They act non trivially on the second subspace $\h'=N(P+Q-1)^\perp$. Denote by $P'$ and $Q'$ the restrictions of $P$ and $Q$ to this subspace. Then the isometric part in the polar decompostion of $P'+Q'-1$ is a selfadjoint unitary operator $V$ which satisfies $VP'V=Q'$, $VQ'V=P'$. We relate these operator with the differential geometry of the space $\p(\h')$  of projections in $\h'$  (or Grassmann manifold of $\h'$). Specifically, with the unique short geodesic curve joining $P'$ and $Q'$ in $\p(\h')$.  For instance, it is shown that $PQ$ is S-decomposable if and only if the velocity operator of the unique geodesic joining $P'$ and $Q'$ is diagonalizable.

In Section 5 it is shown that any contraction $\Gamma\in\b(\h)$ is the $1,1$ entry of a unitary operator times a product of projections acting in $\h\times \h$.

\section{Products and differences of projections}

If $T\in\p\cdot\p$, then $T=P_{\overline{R(T)}}P_{N(T)^\perp}$. This is a result of T. Crimmins (unpublished; there is a proof in \cite{Radjavi-Williams}  Theorem 8). Moreover, Crimmins proved that $T\in\b(\h)$ belongs to $\p\cdot\p$ if and only if $TT^*T=T^2$ \cite{Radjavi-Williams}. However, the factorization $T=P_{\overline{R(T)}}P_{N(T)^\perp}$ is one among among many others. In \cite{cm}, Theorem 3.7, it is proved that if $T\in\p\cdot\p$, then $T=P_\s P_\t$ if and only if 
$$
\overline{R(T)}\subset\s \ , \  N(T)^\perp\subset\t \hbox{  and  } (\s\ominus\overline{R(T)}) \oplus  (\t\ominus N(T)^\perp)\subset R(T)^\perp\cap N(T).
$$
In \cite{cm}, for any $T\in\p\cdot\p$ the set of all pairs $(\s,\t)$ of closed subspaces such that $T=P_\s P_\t$ is denoted by ${\cal X}_T$.  Our first result is a characterization of ${\cal X}_T$ for S-decomposable $T$. The proof is essentially that of Theorem 4.1 in \cite{jot}, where $T$ is supposed to be a compact element of $\p\cdot \p$. We include a proof for the reader's convenience.

\begin{teo}\label{biorto}
Let $\s, \t\subset\h$ be closed subspaces  of $\h$. Then 
$T=P_\s P_\t$ is S-decomposable if and only if there exist orthonormal bases $\{\psi_n: n\ge 1\}$ of $\s$,  $\{\xi_n: n\ge 1\}$ of $\t$  such that $\langle\xi_n,\psi_m\rangle=0$ if $n\ne m$. In such case, the numbers  $|\langle\xi_n, \psi_n\rangle|$ are the singular values of $T$.
\end{teo}
\begin{proof}
Suppose that $\{\psi_n\}$, $\{\xi_n\}$ are orthonormal bases of $\s$, $\t$ respectively, such that 
$$
\langle\psi_n,\xi_m\rangle=0 \ \hbox{ for } n\ne m.
$$
Therefore
$$
P_\s P_\t=(\sum_{n\ge 1} \langle\  \ \ , \psi_n\rangle\psi_n)(\sum_{m\ge 1} \langle\ \ \ , \xi_m\rangle\xi_m)=\sum_{n\ge 1}\langle\psi_n,\xi_n\rangle\psi_n\otimes\xi_n.
$$
In order to get the Schmidt decomposition of $P_\s P_\t$, we only need to replace  $\langle\psi_n,\xi_n\rangle$ by the appropriate sequence of positive numbers: write $\langle\psi_n,\xi_n\rangle=e^{i\theta_n}|\langle\psi_n,\xi_n\rangle|$ and replace $\psi_n$ by $\psi'_n=e^{-i\theta_n}\psi_n$. Then $\{\psi'_n\}$ is still an orthonormal basis of $\s$, and 
$$
\langle\psi'_n,\xi_n\rangle=|\langle\psi_n,\xi_n\rangle|=s_n
$$ 
are the singular values in the decomposition
$$
P_\s P_\t=\sum_{n\ge 1}|\langle\psi_n,\xi_n\rangle|\psi'_n\otimes\xi_n.
$$
This shows that $P_\s P_\t$ is S-decomposable.

Conversely, if $T=P_\s P_\t$ is S-decomposable it has a singular value decomposition 
$$
T=\sum_{n\ge 1} s_n \psi_n\otimes \xi_n
$$
and it holds that $T^2=TT^*T$. Then 
$$
T^*=\sum_{n\ge 1} s_n \langle\ \  ,\psi_n\rangle \xi_n \ \ ,  \ TT^*T=\sum_{n\ge 1} s^3_n \langle\ ,\xi_n\rangle \psi_n,
$$ 
and
$$
T^2=\sum_{n,m\ge 1} s_ns_m \langle\psi_n,\xi_n\rangle\langle\ ,\xi_n\rangle\psi_n=\sum_{n\ge 1}s_n \left( \sum_{m\ge 1} \langle\ ,s_m\langle\xi_n,\psi_m\rangle\xi_m\right) \psi_n.
$$
Using $TT^*T=T^2$ we get, for each $n\ge 1$
$$
\sum_{m\ge 1} s_ns_m\langle\xi_n,\psi_m\rangle\xi_m=s_n^3\xi_n.
$$
Then  $\langle\xi_n,\psi_m\rangle=0$ if $n\ne m$ and $s_n=\langle\xi_n,\psi_n\rangle$. Finally, we can extend the orthonormal bases $\{\psi_n\}$ of $\overline{R(T)}$ and $\{\xi_n\}$ of $N(T)^\perp$ to orthonormal bases of $\s$ and $\t$. In fact, if $\psi\in \s\ominus \overline{R(T)}$ and $\xi\in\t\ominus N(T)^\perp$, then 
$$
\langle\psi,\xi\rangle=0,
$$
because  
$$
(\s\ominus\overline{R(T)}) \oplus  (\t\ominus N(T)^\perp)\subset R(T)^\perp\cap N(T).
$$
\end{proof}

Next, we show that $T=PQ$ is S-decomposable if and only if $A=P-Q$ is diagonalizable, and establish the relation between the singular values of $T$ and the eigenvalues of $A$. We present this equivalence as two separate theorems, to avoid  too long a statement.

\begin{teo}\label{A diagonal}
Suppose that $T=PQ$ is S-decomposable with singular values $s_n$. Then $A=P-Q$ is diagonalizable, with eigenvalues $\pm(1-s_n^2)^{1/2}$, $n\ge 1$, plus, eventually, $0, -1$ and $1$.
\end{teo}
\begin{proof}
Put as above $T=\sum_{n\ge 1} s_n \psi_n\otimes \xi_n$, with $\xi_n\in R(Q)$ and $\psi_n\in R(P)$.
First note that $s_n\le s_1=\|T\|\le\|P\|\|Q\|\le 1$. Moreover, $s_1=1$ means that $T\xi_1=\eta_1$ and thus 
$\|P(Q\xi_1)\|=1=\|\xi_1\|\ge \|Q\xi_1\|\ge \|P(Q\xi_1)\|$, i.e., $\xi_1\in R(Q)$and  $Q\xi_1=\xi_1\in R(P)$. Then $\xi_1=\psi_1$. The same happens for all $n$ such that $s_n=1$: the associated vectors $\xi_n=\psi_n$ generate $R(P)\cap R(Q)$. Note that $A=P-Q$ is trivial in this subspace. 

Suppose that $s_k<1$. Apparently, 
$$
A\xi_k=P\xi_k-Q\xi_k=PQ\xi_k-\xi_k=T\xi_k-\xi_k=s_k\psi_k-\xi_k
$$ 
and 
$$
A\psi_k=P\psi_k-Q\psi_k=P\psi_k-QP\psi_k=\psi_k-T^*\psi_k=\psi_k-s_k\xi_k.
$$
Then 
$$
A^2\xi_k=(1-s^2_k)\xi_k . \  \  A^2\psi_k=(1-s_k)^2\psi_k.
$$
Since $s_k=|\langle\xi_k,\psi_k\rangle|<1$ and $\|\xi_k\|=\|\psi_k\|=1$, it follows that $\xi_k,\psi_k$ span a two-dimensional eigenspace for $A^2$, with eigenvalue $1-s_k^2$.
Then 
$$
\nu_k=((1-s_k^2)^{1/2}-1)\xi_k +s_k\psi_k \hbox{ and } \omega_k=(-(1-s_k^2)^{1/2}-1)\xi_k+s_k\psi_k
$$ 
are orthogonal eigenvectors for $A$, with eigenvalues $(1-s_k^2)^{1/2}$ and $-(1-s_k^2)^{1/2}$, respectively.

The orthogonal systems $\xi_k$ and $\psi_k$ can be extended to orthonormal bases of $R(P)$ and $R(Q)$, respectively (as in the proof of Theorem \ref{biorto}). On the extension of the system $\xi_k$, i.e.,  $R(P)\ominus R(T)$, $A=P-Q$ equals $1$. On the extension of $\psi_k$, $R(Q)\ominus N(T)^\perp$, $A$ equals $-1$.  Together, these extended systems span $R(P)+R(Q)$, and here $A$ is diagonalizable. On the orthogonal complement of this subspace, namely $N(P)^\perp\cap N(Q)^\perp$, $A$ is trivial.
\end{proof}
\begin{rem}
Note that, except for $1$ and $-1$, the eigenvalues $(1-s_k^2)^{1/2}$ and $-(1-s_k^2)^{1/2}$ of $A$ have the same multiplicity. Also note that 
$$
N(A-1)=R(P)\cap N(Q) ,  \ \  N(A+1)=N(P)\cap R(Q),
$$
and $N(A)=R(P)\cap R(Q) \oplus N(P)\cap N(Q)$.
\end{rem}
The above result has a converse. In \cite{davis} Chandler Davis proved  that operators $A=P-Q$ are characterized as follows: in the generic part of $A$,
namely
$$
\h_0=\{N(A)\oplus N(A-1) \oplus N(A+1)\}^\perp,
$$
which reduces $P,Q$ and $A$, if we denote $P_0=P|_{\h_0}$, $Q_0=Q|_{\h_0}$ and
$$
A_0=A|_{\h_0}=P_0-Q_0,
$$
there exists a symmetry $V$ ($V^*=V^{-1}=V$) such that $VA=-AV$ and 
$$
P_0=P_V=\frac12\{1+A_0+V(1-A_0^2)^{1/2}\} \ \ , \ \  Q_0=Q_V=\frac12\{1-A_0+V(1-A_0^2)^{1/2}\}.
$$
$V$ is characterized by these properties. 
With these notations we have:
\begin{teo}\label{1.4}
If $A=P-Q$ is diagonalizable with (non nil) eigenvalues $\pm\lambda_n$ ($0<|\lambda_n|<1$) and $\pm 1$ , then $T=PQ$ is S-decomposable with singular values $(1-\lambda_n^2)^{1/2}$ and $1$.
\end{teo}
\begin{proof}
On the non generic parts $N(A-1) \oplus N(A+1)$, $T$ equals zero. In $N(A)=R(P)\cap R(Q) \oplus N(P)\cap N(Q)$, $T$ is
$$
1\oplus 0.
$$
Thus $PQ$ is diagonal (thus S-decomposable) in $\h_0^\perp$. In $\h_0$, after straightforward computations (note that $V$ commutes with $A_0^2$) one has
$$
P_0Q_0=P_VQ_V=\frac{V}{2}\{V(1-A_0^2)^{1/2}+1-A_0\}(1-A_0^2)^{1/2}.
$$
Since $A_0$ is diagonalizable, and there exists the symmetry $V$ associated to $P_0$ and $Q_0$, which intertwines $A_0$ with $-A_0$, it follows that $A_0$ is of the form
$$
A_0=\sum_{n\ge 1} \lambda_n(E_n-F_n),
$$
where $E_n, F_n$ $(n\ge 1$) are pairwise orthogonal projections with $\dim R(E_n)=\dim R(F_n)=m_n\le \infty$. The eigenvalues $\lambda_n$ of $A_0$ are different from $\pm 1$, because $N(A_0\pm 1)=\{0\}$.  Fix an orthonormal basis $\{\nu^n_k: 1\le k\le m_n\}$ for $R(E_n)$. The fact that $VA=-AV$ implies that $V$ maps (the $\lambda_n$-eigenspace) $R(E_n)$ onto  (the $-\lambda_n$-eigenspace) $R(F_n)$, and back. Then we can consider for $R(F_n)$ the orthonormal basis given by $\omega^n_k=V\nu^n_k$. Thus also $V\omega_k^n=\nu_k^n$. Then 
$$
P_0Q_0\nu_k^n=\frac12(1-\lambda_n^2)\nu_k^n +\frac12 (1-\lambda_n)(1-\lambda_n^2)^{1/2}\omega_k^n
$$
and
$$
P_0Q_0\omega_k^n=\frac12(1-\lambda_n^2)\omega_k^n +\frac12 (1+\lambda_n)(1-\lambda_n^2)^{1/2}\nu_k^n.
$$
It follows that the $2$-dimensional subspace generated by (the orthonormal vectors) $\nu_k^n$ and $\omega_k^n$ is  invariant for $P_0Q_0$. The matrix of $P_0Q_0$ restricted to this subspace (in this basis) is
$$
 \frac12\left( \begin{array}{cc} (1-\lambda_n^2) & (1+\lambda_n)(1-\lambda_n^2)^{1/2}
 \\ (1-\lambda_n)(1-\lambda_n^2)^{1/2} & (1-\lambda_n^2) \end{array} \right),
 $$
 whose singular values are $0$ and $(1-\lambda_n^2)^{1/2}$. In the orthonormal basis $\{\nu_k^n, \omega_k^n\}$ of $\h_0$   (paired in this fashion), the operator $P_0Q_0$ is block-diagonal, with $2\times 2$ blocks.  It follows that $PQ$ is S-decomposable with singular values $(1-\lambda_n^2)^{1/2}$ and, eventually, $1$.
 The singular value $1$ occurs only if $R(P)\cap R(Q)\ne \{0\}$.
\end{proof}
\begin{rem}
The multiplicity of $(1-\lambda^2_n)^{1/2}$ as a singular value of  $PQ$  is $m_n$.
\end{rem}

\begin{rem}
From the above results, which relate  eigenvalues of $P-Q$ and singular values of $PQ$, it follows that if $PQ$ is compact, and either $P$ or $Q$ have infinite rank,  then $\|P-Q\|=1$. Indeed, if $PQ$ is compact, the singular values accumulate eventually at $0$, and therefore the eigenvalues of $A$ accumulate at $1$. However, this result holds with more generality.  It is a simple exercise that if $p\ne q$ are non nil projections in a C$^*$-algebra such that $pq=0$, then $\|p-q\|=1$. Our case consists in reasoning in the Calkin algebra: $p=\pi(P)$, $q=\pi(Q)$, where $\pi:\b(\h)\to \b(\h)/\k(\h)$ is the quotient homomorphism. Then 
$$
1\le \|p-q\|\le \|P-Q\|\le 1.
$$
\end{rem}
The following result will be useful to provide further examples. In a special case (see Example \ref{ejemplo1} in Section 3), it was proven  by M. Smith (\cite{smith}, Th. 3.1) 

\begin{prop}\label{lechealgato}
$PQ$ is S-decomposable if and only if $P(1-Q)$ is S-decomposable (and therefore if and only if $(1-P)Q$ or $(1-P)(1-Q)$ are S-decomposable).
\end{prop}
\begin{proof}
$P(1-Q)$ is S-decomposable if and only if $P(1-Q)P=P-PQP$ is diagonalizable. This operator acts non trivially only in $R(P)$. Thus,  it is diagonalizable  if and only if it is diagonalizable in $R(P)$. Adding $1-P$ (equal to the identity in $N(P)$), one obtains that this latter fact is equivalent to 
$1-PQP=1-P \oplus P-PQP$ being diagonalizable in $\h=N(P)\oplus R(P)$. Clearly $1-PQP$ is diagonalizable if and only if $PQP$ also is, i.e., if and only if $PQ$  is S-decomposable.
\end{proof}
As a direct consequence of this fact, one obtains the following corollary
\begin{coro}
Let $P,Q$ be projections. Then $P-Q$ is diagonalizable if and only if $P+Q$ is diagonalizable.
In that case, $\lambda_n$ is an eigenvalue of $P-Q$ with $0<|\lambda_n|<1$  if and only if $1\pm (1-\lambda_n)^2$ is an eigenvalue of $P+Q$, with the same multiplicity. 
\end{coro}
\begin{proof}
By the above results, any eigenvalue $\lambda_n=\pm(1-s_n^2)^{1/2}$, where $s_n$ is a singular value of $PQ$, or equivalently, $s_n^2$ is an eigenvalue of $PQP$. On the other hand, from the proof of Proposition \ref{lechealgato}, the eigenvalues of 
$$
1-PQP=1-P\oplus PQ^\perp P
$$
are $1$, and $1-s_n^2$. Then again by Theorem \ref{A diagonal}, the eigenvalues of  $P-Q^\perp=P+Q-1$ are $\pm s_n$, and thus the eigenvalues of $P+Q$ are $1\pm s_n=1\pm (1-\lambda_n^2)^{1/2}$. Since $P-Q$ is a difference of projections, the eigenvalues  $+\lambda$ and $-\lambda$ (when $0<|\lambda|<1$) have the same multiplicity (see \cite{pemenoscu}), and by the above results, these add up to the multiplicity of $s=(1-\lambda^2)^{1/2}$ as a singular value of $PQ$. This number clearly equals the multiplicity of $(1-s^2)^{1/2}$ as a singular value of $PQ^\perp$. Note that $P+Q-1=P-Q^\perp$ is also a difference of projections, therefore the multiplicities of $\pm s=\pm (1-\lambda^2)^{1/2}$ coincide ($0<s<1$).  
\end{proof}
\begin{rem}
The multiplicity of $1$ as an (eventual) eigenvalue of $P-Q$ is the dimension of $R(P)\cap N(Q)$, the multiplicity of $-1$ is the dimension of $N(P)\cap R(Q)$, the sum of these multiplicities is the multiplicity of $0$ in $P-Q^\perp$, or the multiplicity of $1$ in $P+Q$. Similarly, the multiplicity of $0$ in $P-Q$ equals the sum of the multiplicities of $0$ and $2$ in $P+Q$.
\end{rem}

\begin{rem}
To study the examples in the next section, it will also be useful to note that if $P$ has infinite rank and $PQ$ is compact, then $P(1-Q)$ is S-decomposable but non compact.
\end{rem}

\section{Examples}
 
\begin{ejem}\label{ejemplo1}

Let $I,J\subset \mathbb{R}^n$ be Lebesgue-measurable sets of finite measure. Let $P_I,Q_J$ be the projections in $L^2(\mathbb{R}^n,dx)$ given by
$$
P_If=\chi_If \ \ \hbox{ and } \ \ Q_Jf= \left(\chi_J \hat{f}\right)\check{\ },
$$
where $\chi_L$ denotes the characteristic function of the set $L$.
Equivalently, denoting by $U_\f$ the Fourier transform regarded as a unitary operator acting in $L^2(\mathbb{R}^n,dx)$, then
$$
P_I=M_{\chi_I} \ \hbox{ and } Q_J=U_\f^*M_{\chi_J}U_\f .
$$
In  \cite{donoho} (Lemma 2) it is proven that $P_IQ_J$ is a Hilbert-Schmidt operator. See also \cite{folland}. 
Then $T=P_IQ_J$ is S-decomposable (with square summable singular values)
These products play a relevant role in operator theoretic formulations of the uncertainty principle \cite{donoho}, \cite{folland}.

In this case one has the spectral picture of $A=P_I-Q_J$. It is known \cite{lenard}, \cite{folland} that 
$$
N(P_I)\cap N(Q_J)=R(P_I)\cap N(Q_J)=N(P_I)\cap R(Q_J)=\{0\}
$$
and $R(P_I)\cap R(Q_J)$ is infinite dimensional. Thus $N(A)=R(P_I)\cap R(Q_J)$ is infinite dimensional, $N(A\pm 1)=\{0\}$, and the eigenvalues of $A$ are of the form $\pm (1-s_k^2)^{1/2}$, where the sequence $s_k$ belongs to $\ell^2(\mathbb{Z})$. In special cases, e.g. $I=[0,T], J=[-\Omega,\Omega]$ intervals in $\mathbb{R}$, the eigenfunctions are known and the eigenvalues have multiplicity one \cite{bell labs}.

If one relaxes the condition that the sets be of finite measure, $P_IQ_J$ ceases to be compact.
Using Proposition \ref{lechealgato}, one obtains non compact examples: replacing the above conditions by $|\mathbb{R}^n\setminus I|<\infty$ or 
 $|\mathbb{R}^n\setminus J|<\infty$ (see also \cite{smith}, one obtains non-compact, S-decomposable products of projections.

Note also that, due to Theorem \ref{biorto}, in the above cases (i.e. both $I$ and $J$ have finite or co-finite measure), the subspaces 
$R(P_I)=\{f\in L^2(\mathbb{R}^n): f|_{\mathbb{R}^n\setminus I}=0\}$ and $R(Q_J)=\{g\in L^2(\mathbb{R}^n): \hat{g}|_{\mathbb{R}^n\setminus J}=0\}$ have orthonormal bases $\{f_n\}$ and $\{g_n\}$, repectively, which satisfy $\langle f_n,g_m\rangle=0$ if $n\ne m$.

We study more carefully the case
$$
I=[0,+\infty) , \  \ J=[-1,1],
$$
not covered above.
Straightforward computations (see \cite{lenard}) show that the operator $P_IQ_J$, acting in $L^2(0,+\infty)$ is given by
$$
P_IQ_JP_If(x)=\frac{1}{\pi}\int_0^\infty \sinc(x-t)f(t) dt.
$$
Let us prove that $P_IQ_JP_I$ is non compact. For $n\in\mathbb{N}$, let 
$$
e_n(x)=\left\{ \begin{array}{l} e^{-\frac{1}{n} x} e^{ix} , \hbox{ if } x\ge 0 \\ 0 \ \hbox{ otherwise. } \end{array} \right.
$$
Apparently $e_n\in L^2(\mathbb{R})$ and $\|e_n\|_2^2=\frac{n}2$.  Note that
$$
P_IQ_JP_I\frac{e_n(x)}{\|e_n\|_2}=\frac{\sqrt{2}}{\pi\sqrt{n}}\{\int_0^x sinc(x-t)e_n(t)dt + \int_x^\infty sinc(x-t)e_n(t)dt\}.
$$
Changing variables $v=x-t$ in the first integral and $u=t-x$ in the second, one obtains
$$
\frac{\sqrt{2}}{\pi\sqrt{n}}\{e_n(x) \int_0^x sinc(v)e_n(-v)dv + e_n(x) \int_0^\infty sinc(u)e_n(u) du\}.
$$
The  second  integral, which we shall denote $\lambda_n$, can be computed. Denote by $\mathbb{L}$ the usual Laplace transform.
Then
$$
\lambda_n=\int_0^\infty sinc(u)e_n(u) du=\int_0^\infty sinc(u)cos(u)e^{-\frac1{n}u} du+ i  \int_0^\infty sinc(u)\sin(u)e^{-\frac1{n}u} du
$$
$$
=\mathbb{L}(\frac{\sin(t)}{t}\cos(t))|_{t=\frac1{n}}+i\  \mathbb{L}(\frac{\sin^2(t)}{t})|_{t=\frac1{n}}=\frac\pi{2} - \arctan(\frac1{n}) + i\ \frac14 ln(1+4n^2).
$$
Let us denote by $F_n(x)$ the left hand integral,
$$
F_n(x)=\int_o^x sinc(t)e_n(-t)dt.
$$
\begin{lem}
With the current notations, 
$$
\frac1{\|e_n\|_2}{\|P_IQ_JP_Ie_n-e_n\|_2}  \to 0, \hbox{ as } \ \ n\to\infty.
$$
\end{lem}
\begin{proof}

Compute
$$
\langle e_n F_n,e_n\rangle=\int_0^\infty e_n(x)F_n(x)\bar{e}_n(x) dx=\int_0^\infty e^{-\frac2{n}x}F_n(x) dx.
$$
Integrating by parts, and using that (by means of the L'Hospital rule !), we get
$$
\lim_{x\to +\infty} \frac{F_n(x)}{e^{\frac{2}{n}x}}=\lim_{x\to\infty} \frac{n}2 \frac{sinc(x)e_n(-x)}{e^{\frac{2}{n}x}}=\frac{n}2 \lim_{x\to +\infty} \frac{sinc(x)}{e^{\frac1{n}x}}=0,
$$
and that $F_n(0)=0$. Then
$$
\langle e_n F_n,e_n\rangle=\frac{n}2 \int_0^\infty e^{-\frac2{n}x} F'_n(x) dx=\frac{n}2 \int_0^\infty  e^{-\frac2{n}x} sinc(x) e^{\frac1{n}x}e^{-ix}dx
$$
$$
=\frac{n}2 \int_0^\infty  e^{-\frac1{n}x} sinc(x)\cos(x) dx -i \frac{n}2 \int_0^\infty  e^{-\frac1{n}x} sinc(x)\sin(x) dx,
$$
which, by computations similar as above involving the Laplace transform, equals
$$
\frac{n}2\{\frac\pi{2}- \arctan(\frac1{n}) -i\frac14 ln(1+4n^2)\}=\frac{n}2\bar{\lambda}_n.
$$
Then 
$$
\langle P_IQ_JP_I e_{n},e_{n}\rangle=\frac1{\pi}\langle e_{n}(F_{n}+\lambda_{n}),e_{n}\rangle=\frac1{\pi}\{ \lambda_n\|e_n\|_2^2+\langle F_ne_n,e_n\rangle\}
$$
$$
=\frac1{\pi}\{\frac{n}2\lambda_n+\frac{n}2\bar{\lambda}_n\}=\frac{n}{\pi} Re(\lambda_n)=\frac{n}{\pi}\{\frac\pi{2}-\arctan(\frac1{n})\}.
$$
Then
$$
\|P_IQ_JP_Ie_n-e_n\|_2^2=\|P_IQ_JP_Ie_n\|_2^2+\|e_n\|_2^2-2 Re \langle P_IQ_JP_Ie_n,e_n\rangle
$$
$$
\le 2\|e_n\|_2^2-2 \frac{n}{\pi}\{\frac\pi{2}-\arctan(\frac1{n})\}=2\frac{n}{\pi} \arctan(\frac1{n}).
$$
Therefore
$$
\frac1{\|e_n\|_2^2}{\|P_IQ_JP_Ie_n-e_n\|_2^2}\le \frac4{\pi} \arctan(\frac1{n})\to 0.
$$
\end{proof}
\begin{prop}
If $I=[0,+\infty)$ and $J=[-1,1]$, then $P_IQ_JP_I$ is non compact, with 
$$
\|P_IQ_JP_I\|=\|P_IQ_J\|=1 \ \hbox{ and  } \ \|P_I-Q_J\|=1.
$$
\end{prop}
\begin{proof}
If $P_IQ_JP_I$ were compact, there would exist a subsequence $f_k=\frac1{\|e_{n_k}\|_2}e_{n_k}$ such that $P_IQ_JP_If_k$ is convergent. By the above lemma, this would imply that the sequence $f_k$ is convergent. This is clearly not the case. For instance,
$$
\langle f_k,f_{k+1}\rangle=\frac1{\|e_{n_k}\|\|e_{n_{k+1}}\|}\langle e_{n_k},e_{n_{k+1}}\rangle=\frac2{n_k^{1/2}n_{k+1}^{1/2}} \int_0^\infty e^{-(\frac1{n_k}+\frac1{n_{k+1}})x} dx
=\frac{n_k^{1/2}n_{k+1}^{1/2}}{n_k+n_{k+1}}
$$
which is less than $\frac12$ 
by the geometric-arithmetic inequality. This clearly implies that  the sequence of the unit vectors $f_k$ cannot be convergent.

The last assertions follow from the above lemma.
\end{proof}

\begin{rem} Note that  Example \ref{ejemplo1} above shows, in particular, that the Volterra-like integral operator 
$$
Bf(x)=\int_0^x sinc(x-t)f(t) dt
$$
is unbounded in $L^2(0,+\infty)$ (though it is a Volterra operator on any finite interval $[0,r]$, thus compact with trivial spectrum in $L^2(0,r)$, for $r<\infty$). Indeed, if it were bounded, then
$T=P_IQ_JP_I-B$, 
$$
Tf(x)=\int_x^\infty sinc(x-t)f(t) dt
$$
would be bounded. But the computations above show that the functions $e_n(x)=e^{(-\frac1{n} +i)x}$ are eigenfunctions for $T$, with unbounded eigenvalues $\lambda_n$.
\end{rem}

\end{ejem}
\begin{ejem}\label{ejemplo2}
Let $\h=L^2(\mathbb{T},dt)$  
where $\mathbb{T}$ is the $1$-torus, and consider the decomposition 
$$
\h=\h_-\oplus\h_+,
$$
where $\h_+$ is the Hardy space. Let  $\varphi, \psi$ be continuous  functions in $\mathbb{T}$ with $|\varphi(e^{it})|=|\psi(e^{it})|=1$ for all $t$, and
$$
P=P_{\varphi\h_+}^\perp=1-P_{\varphi\h_+} \ , \ \ Q=P_{\psi\h_+}.
$$
Since $\varphi$ and $\psi$ are unimodular, the multiplication operators $M_\varphi$, $M_\psi$ are unitary in $\h$ and thus
$$
PQ=M_\varphi P_- M_{\bar{\varphi}\psi}P_+ M_{\bar{\psi}}.
$$
Note that $P_- M_{\bar{\varphi}\psi}|_{\h_+}=H(\bar{\varphi}\psi)$ is   the Hankel operator with symbol $\bar{\varphi}\psi$,  which is compact by Hartman's theorem \cite{hartman}  (see also Theorem 5.5 in  \cite{peller}). Thus $T=PQ$ is compact, and therefore S-decomposable. 

Again using Proposition \ref{lechealgato}, one obtains non compact S-decomposable examples. For instance, put now
$$
P=P_{\varphi \h_+} , \ \ Q=P_{\psi \h_+}.
$$
In this case
$$
PQ=M_\varphi P_+ M_{\bar{\varphi}\psi}P_+ M_{\bar{\psi}}
$$
is decomposable
Thus the  operator $P_+M_{\bar{\varphi}\psi}P_+$ is non-compact and S-decomposable in $L^2(\mathbb{T})$. Since it acts non trivially in $\h_+$, it follows that the Toeplitz operator $T_{\bar{\varphi}\psi}$ is S-decomposable in $\h_+$.

On the other hand, using Theorem \ref{A diagonal}, it follows that 
$$
A=P_{\varphi\h_+}-P_{\psi\h_+}
$$
diagonalizable. In \cite{grassh2} it was shown that $\pm 1$ are  eigenvalues of $A$ only if the winding numbers of $\varphi$ and $\psi$ do not coincide. The other  eigenvalues of $A$  are $\pm(1-s_n^2)^{1/2}$, where $s_n$ are the singular values of $T_{\bar{\varphi}\psi}$, and $0$. Since this operator has closed range (being a Fredholm operator), the eigenvalues do not accumulate at $\pm 1$. The nullspace of $A$ is infinite dimensional, it contains the subspace $\varphi\psi\h_+$.

Again, using Theorem \ref{biorto}, one obtains that, with the above hypothesis on $\varphi$ and $\psi$, there exist orthonormal bases $\{f_n\}$ and $\{g_n\}$ of $\h_+$ such that $\langle\varphi f_n,\psi g_m\rangle=0$ if $n\ne m$.

In \cite{howland} (Theorem 5.2) J.S. Howland proved that if the function $f$ on $\mathbb{T}$ is $C^2$  on the complement of a finite set $\{z_1, \dots, z_n\}$ at which the lateral limits $f(z_i^\pm)$ and $f'(z_i^\pm)$ exist, and one defines the jump of $f$ at $z$ to be 
$$
j(z)=f(z^+)-f(z^-),
$$
then the absolutely continuous part of the Hankel operator $H(f)$ is unitarily equivalent to 
$$
\oplus_{i=1}^n M_{i,z}
$$
where $M_{i,z}$ denotes the operator of multiplication by the variable $z$ in  $L^2(-\frac12 j(z_i),\frac12 j(z_i))$. In particular, this implies that if $\bar{\varphi}\psi$ is piecewise $C^2$ with jumps as $f$ above, then $PQ$ can be decomposed as a finite direct sum of operators, some of which are multiplication by the variable in $L^2$ of an interval. Clearly these operators are not S-decomposable. Then $PQ$ is  not  S-decomposable.
\end{ejem}

\begin{ejem}\label{E}
Let $\h=\l\times\s$,  $B:\s\to \l$ a bounded operator, and  $E=E_B$  the idempotent operator given by the matrix
$$
E=\left( \begin{array}{cc}  1_\l & B \\ 0 & 0 \end{array} \right).
$$
Any idempotent in $\b(\h)$ can be expanded in this form. In \cite{clasesE} the reader can find a study of the properties of $E$ in terms of those of $B$.  
Consider $P=P_{R(E)}=P_\l$ and $Q=P_{N(E)}$ and $T=PQ$. Straightforward computations show that $R(E)=\l$ and that
$$
P_{N(E)}=(1-E)(1-E-E^*)^{-1}=\left( \begin{array}{cc}  BB^*(1+BB^*)^{-1} & -B(1+B^*B)^{-1} \\ -B^*(1+BB^*)^{-1} & (1+B^*B)^{-1} \end{array} \right). 
$$
Then  
$$
TT^*=PQP=\left( \begin{array}{cc}   BB^*(1+BB^*)^{-1} & 0 \\ 0 & 0 \end{array} \right).
$$
Apparently, $T$ is S-decomposable if and only if $BB^*(1+BB^*)^{-1}$ is diagonalizable, which is equivalent to $BB^*$ being diagonalizable, or $B$ S-decomposable. Note also that $T$ is compact if and only if $B$ is compact. 

If one applies Theorem \ref{biorto} to this example, one obtains that $B$ is S-decomposable if and only if there exist orthonormal bases $\{(0,v_n)\}$ of $\{0\}\times \l$ and $\{(w_n,Bw_n)\}$ of the graph of $B$, such that $\langle(0,v_n),(w_m,Bw_m)\rangle=\langle v_n,Bw_m\rangle=0$ if $n\ne m$. This fact can be proved straightforwardly. 
\end{ejem}

\section{Moore-Penrose pseudoinverses}

Penrose  \cite{penrose} and Greville \cite{greville} proved that , for $n\times n$ square matrices, the Moore-Penrose inverse of an idempotent matrix $E$ is a product of orthogonal projections. More precisely, it holds that
$$
E^\dagger=P_{N(E)^\perp}P_{R(E)}.
$$
Since for matrices $(A^\dagger)^\dagger=A$, Penrose-Greville theorem can be stated as follows: an $n\times n$  matrix $E$ is idempotent if and only if $E^\dagger$ is a product of two orthogonal projections. This result was extended to infinite dimensional Hilbert space operators in \cite{cmanterior} , provided that $PQ$ is supposed to have closed range. In the case that $R(PQ)$  is not closed, there is still a similar characterization, but one needs to define the Moore-Penrose inverse for certain unbounded operators. The reader is referred to \cite{cm}.
As in example \ref{E}, if $E$ is an idempotent operator, in terms of the decomposition $\h=R(E)\oplus R(E)^\perp$, one has
$$
E=\left( \begin{array}{cc} 1 & B \\ 0 & 0 \end{array} \right),
$$
where $B:R(E)^\perp\to R(E)$.

Combining the above facts and previous results we obtain the following:
\begin{coro}
Let $E\in\b(\h)$ be an idempotent operator. Then the following are equivalent:
\begin{enumerate}
\item
$E$ is S-decomposable
\item
$B$ is S-decomposable
\item
$P_{N(E)^\perp}P_{R(E)}$ is S-decomposable.
\item
$P_{N(E)}P_{R(E)}$ is S-decomposable.
\item
$P_{R(E)}-P_{N(E)}$ is diagonalizable.
\item
$P_{R(E)}+P_{N(E)}$ is diagonalizable.
\item
There exist orthonormal bases $\{\eta_n\}$ of $R(E)$ and $\{\nu_n\}$ of $N(E)$ such that $\langle\eta_n,\nu_m\rangle=0$ if $n\ne m$.
\end{enumerate}
\end{coro}
 
 Some of these conditions were proven in \cite{clasesE}.

\begin{rem}
By a theorem by Buckholtz (\cite{buckholtz}, Theorem 1), since $\h$ is the direct sum of $R(E)$ and $N(E)$, it follows that $P_{R(E)}-P_{N(E)}$ is invertible for every idempotent $E$, which in turn implies that $P_{R(E)}+P_{N(E)}$ is invertible. In fact, for any $P,Q\in\p(\h)$, $P-Q$ is invertible if and only if $\|PQ\|<1$ and $\|(1-P)(1-Q)\|<1$, while $P+Q$ is invertible if and only if $\|(1-P)(1-Q)\|<1$. In geometric terms, $\|PQ\|$  is the cosine of the (Dixmier) angle between $R(P)$ and $(Q)$, and  $\|(1-P)(1-Q)\|$ is the  cosine of the angle between $N(P)$ and $N(Q)$. If $\h$ is the direct sum of $R(P)$ and $RQ)$, these angles coincide and are not zero.
\end{rem}
Finally, note that if   $T$ is S-decomposable with expansion $T=\sum_{n\ge 1} \s_n\langle\ \ \ , \xi_n\rangle\psi_n$, then
$$
T^\dagger=\sum_{n\ge 1} \frac{1}{s_n}\langle\ \ \ , \psi_n\rangle\xi_n.
$$
\section{Isometries}
Given a subspace $\s\subset\h$ with a given orthonormal basis ${\bf  B}_\s=\{\xi_n: b\ge 1\}$, an isometry is defined,
$$
X_{{\bf  B}_\s}:\ell^2\to \h, \ \  X_{{\bf  B}_\s}(\{x_n\})=\sum_{n\ge 1} x_n \xi_n,
$$
whose range is $\s$.  Observe that, by definition, the set of all S-decomposable operators in $\h$ can be described as
$$
\{XDY^*: X,Y \hbox{ isometries } \ell^2\to \h, \ D\in\b(\ell^2) \hbox{ diagonal with positive entries}\}.
$$
The condition of bi-orthogonality of Theorem \ref{biorto} can be written in terms of the corresponding isometries.
\begin{prop}
Let $\s,\t$ be closed subspaces of $\h$. Then $T=P_\s P_\t$ is S-decomposable if and only if there exist  isometries $X, Y:\ell^2\to \h$, with range $\s$ and $\t$, respectively, such that
$$
X^*Y\in\b(\ell^2)
$$
is a diagonal matrix.
\end{prop}
 \begin{proof}
Suppose that $T$ is decomposable, then by Theorem (\ref{biorto}), there exist orthonormal bases ${\bf  B}_\s=\{\xi_k: k\ge 1\}$ and ${\bf  B}_\t=\{\psi_n: n\ge 1\}$ of $\s$ and $\t$  such that $\langle\xi_n,\psi_k\rangle=0$ if $n\ne k$. Consider the  isometries
$$
X=X_{{\bf  B}_\s} \ \hbox{ and } \ \ Y=X_{{\bf  B}_\t}.
$$
Then
$$
X^*Y(\{x_n\})=\{\langle\psi_n,\xi_n\rangle x_n\},
$$
i.e. $X^*Y$ is a diagonal matrix whose entries are $\langle\psi_n,\xi_n\rangle$.

Conversely, suppose that $X,Y:\ell^2\to\h$ are  isometries with $R(X)=\s$ and $R(Y)=\t$, such that $X^*Y$ is a diagonal matrix. Denote by $\{e_n: n\ge 1\}$ the canonical basis of $\ell^2$. Then $\xi_n=X(e_n)$ and $\psi_k=Y(e_k)$ form orthonormal bases of $\s$ and $\t$. Moreover
$$
\langle\xi_n,\psi_k\rangle=\langle X(e_n),Y(e_k)\rangle=\langle e_n, X^*Y(e_k)\rangle=0 \hbox{ if } n\ne k.
$$
\end{proof}
 
\section{Davis' symmetry}

Let $P,Q$ be projections, and consider 
$$
\h'=\{R(P)\cap N(Q) \oplus N(P)\cap R(Q)\}^\perp.
$$
This subspace reduces $P$ and $Q$, denote by $P'=P|_{\h'}$ and $Q'=Q|_{\h'}$, as operators acting in $\h'$. Note that 
$$
N(P+Q-1)=N(P-(1-Q))=R(P)\cap N(Q) \oplus N(P)\cap R(Q),
$$
and thus $S'=P'+Q'-1$ is a selfadjoint operator with trivial kernel (and thus dense range) in $\h'$. Let 
$$
S'=V|S'|
$$
be the polar decomposition. It follows that $V$ is a selfadjoint unitary operator, i.e., a symmetry. The fact that
$$
S'P'=Q'P'=Q'S' \ \ (\hbox{also } S'Q'=P'Q'=P'S')
$$
implies that the symmetry $V$ intertwines $P'$ and $Q'$:
$$
VP'V=Q' , \  VQ'V=P'.
$$
Also one recovers $P'$ and $Q'$ in terms of $V$ and the difference $A=P'-Q'$, by means of the formulas of the previous section:
$$
P'=P_V , \ \ Q'=Q_V.
$$
These facts were proved by Chandler Davis in \cite{davis}. 
Then $T=PQ$, in the decomposition $\h=\h'^\perp\oplus \h'$ is given by
$$
T=0\oplus VQ'VQ'=0\oplus P'VP'V.
$$
The following result is a straightforward consequence of the results in the previous section:
\begin{prop}
 $T=PQ$ is S-decomposable if and only  $Q'VQ'$ is diagonalizable (equivalently: $P'VP'$ is diagonalizable). If $\{\xi_n\}$ is an orthonormal system of eigenvectors for $Q'VQ'$, then $\langle V\xi_n,\xi_k\rangle=0$ if $n\ne k$.
\end{prop}
\begin{proof}
If $Q'VQ'=\sum_{n\ge 1} \lambda_n \langle\ \ \ , \xi_n\rangle\xi_n$, then 
$$
P'Q'=VQ'VQ'=\sum_{n\ge 1} \lambda_n \langle\ \ \ , \xi_n\rangle V\xi_n,
$$
and thus the orthonormal systems $\{\xi_n\}$ and $\{V\xi_n\}$ are bi-orthogonal.
\end{proof}

\begin{rem}
Suppose that 
$$
P'Q'=VQ'VQ'=\sum_{n\ge 1} s_n \langle\ \ \ , \xi_n\rangle\psi_n.
$$
Then 
$$
Q'VQ'=\sum_{n\ge 1} s_n \langle\ \ \ , \xi_n\rangle V\psi_n=\sum_{n\ge 1} s_n \langle\ \ \ , V\psi_n\rangle\xi_n.
$$
In particular, if all the  singular values have multiplicity $1$, then $V\psi_n=\pm\xi_n$.
\end{rem}

Davis' symmetry is related to the metric geometry of the set $\p(\h)$ of projections in $\h$ (also called Grassmannian manifold of $\h$).
If one measures the length of a continuous piecewise smooth curve $p(t)\in\p(\h)$, $t\in I$, by means of
$$
\ell(p)=\int_I \|\frac{d}{dt} p(t)\| dt, 
$$
it was shown (\cite{pr}, \cite{cpr}) that curves in $\p(\h)$ of the form
$$
P(t)=e^{itX}Pe^{-itX}
$$
for $X^*=X$ with $\|X\|\le\pi/2$, such that $X$ is $P$-codiagonal (i.e $PXP=P^\perp XP^\perp=0$) have minimal length along their paths for $|t|\le 1$. That is, any curve joining a pair of projections in this path cannot be shorter that the part of $P(t)$ which joins these projections. Given two projections $P,Q$, in  \cite{pemenoscu} it was shown that there exists  a unique $X$ ($X^*=X$, $\|X\|\le \pi/2$, $X$ is $P$-codiagonal) such that $e^{iX}Pe^{-iX}=Q$   if and only if 
$$
N(P+Q-1)=\{0\}.
$$
Let us denote $X=X_{P,Q}$ if such is the case. Also in \cite{pemenoscu} it was shown that $V$ and $X_{P,Q}$ are related by
\begin{equation}\label{ecuacionVX}
V=e^{iX_{P,Q}}(2P-1).
\end{equation}
Note that since  (always in the case $N(P+Q-1)=\{0\}$) $\|X_{P,Q}\|\le \pi/2$, $X_{P,Q}$ is obtained from $V$ by means of the usual $\log$ function:
$$
X_{P,Q}=-i\log(V(2P-1)).
$$

Define the geodesic distance $d(P,Q)$  in $\p(\h)$ as
$$
d(P,Q)=\inf\{\ell(p): p \hbox{ joins } P \hbox{ and } Q \hbox{ in } \p(\h)\},
$$ 
Porta and Recht proved in \cite{pr}  that
\begin{equation}\label{distancia geodesica}
d(P,Q)=\|X_{P,Q}\|.
\end{equation}
\begin{rem}
Formula (\ref{ecuacionVX}) has a geometric interpretation. The fact that $X_{P,Q}$ is $P$-codiagonal, is equivalent to saying that  $X_{P,Q}$ and $2P-1$ anti-commute, it follows that $e^{itX_{P,Q}}(2P-1)=(2P-1)e^{-itX_{P,Q}}$. Then, in particular
$V=e^{\frac{i}2 X_{P,Q}}(2P-1)  e^{-\frac{i}2 X_{P,Q}}$, or equivalently
$$
\frac12 (1+V)=e^{\frac{i}2 X_{P,Q}} P e^{-\frac{i}2 X_{P,Q}}.
$$
In other words, the projection $\frac12(1+V)$ (onto the eigenspace where the symmetry  $V$ acts as the identity) is the midpoint of the geodesic $P(t)$ joining $P$ and $Q$.
\end{rem}
From the above facts, the following is apparent:
\begin{coro}
Let $P,Q$ be projections and, as above, $P',Q'$ the respective reductions to $N(P+Q-1)^\perp$, and let  $V$ be Davis' symmetry induced by these.
Then 
$$
P'VP'=P'e^{X_{P',Q'}}P'\ \  \hbox{ and } \ \  Q'VQ'=Q'e^{-X_{P',Q'}}Q'.
$$
Thus $PQ$ is S-decomposable if and only if $P'e^{X_{P',Q'}}P'$ is diagonalizable.
\end{coro}
\begin{proof}
Since $V=e^{iX_{P',Q'}}(2P'-1)$, then
$$
P'VP'=P'e^{iX_{P',Q'}}(2P'-1)P'=P'e^{iX_{P',Q'}}P'.
$$
Similarly, $V=P'e^{iX_{P',Q'}}=e^{-iX_{P',Q'}}Q'$, and so forth.
\end{proof}
\begin{rem}
Since  $Q'=e^{iX_{P',Q'}}P'e^{-iX_{P',Q'}}(2P'-1)P'$, it also follows that 
$$
P'e^{iX_{P',Q'}}P'=P'Q'e^{-iX_{P',Q'}}=e^{-iX_{P',Q'}}Q'P'
$$
and
$$
Q'e^{-iX_{P',Q'}}Q'=Q'P'e^{iX_{P',Q'}}= e^{iX_{P',Q'}}P'Q'.
$$
\end{rem}

\begin{rem}
If the matrix of $X_{P',Q'}$ in terms of $P'$ is given by
$$
X_{P',Q'}=\left( \begin{array}{cc} 0 & Z \\ Z^* & 0 \end{array} \right),
$$
then 
$$
P'VP'=P'e^{iX_{P',Q'}}P'=\left( \begin{array}{cc} \cos(|Z^*|) & 0 \\ 0 & \cos(|Z|) \end{array} \right).
$$
\end{rem}
From this last remark, it follows that
\begin{teo}
$PQ$ is S-decomposable if and only if $Z$ is S-decomposable, if and only if $X_{P',Q'}$ is diagonalizable.
\end{teo}
\begin{proof}
$$
X_{P',Q'}^2=\left( \begin{array}{cc} ZZ^* & 0 \\ 0 & Z^*Z \end{array} \right),
$$
Thus $X_{P',Q'}$ is diagonalizable if and only if $Z$ is S-decomposable. Indeed, if $Z$ is S-decomposable, 
$$
Z=\sum_{n\ge 1} s_n \langle\ \ , v_n\rangle w_n , \ Z^*=\sum_{n\ge 1} s_n \langle \ \ , w_n\rangle v_n.
$$
Note that $\{v_n\}$ span $R(P')$ and $\{w_n\}$ span $R(P')^\perp$, therefore, they are pairwise orthogonal systems of vectors. Then 
$$
X_{P',Q'}v_n=s_nw_n \ \hbox{ and } \ \ X_{P',Q'}w_n=s_nv_n.
$$
For each fixed $n$, the two dimensional space generated by $v_n$ and $w_n$ reduces $X_{P',Q'}$. As in a previous argument, $X_{P',Q'}$ can be diagonalized in each of these spaces, providing a diagonalization of the whole operator $X_{P',Q'}$.  The converse statement is apparent.
\end{proof}

Finally, let us further exploit formula (\ref{ecuacionVX}). 

\begin{coro}
If $A'=P'-Q'$, then
\begin{equation}\label{ecuacionAX}
e^{iX_{P',Q'}}=VA'+(1-A'^2)^{1/2}.
\end{equation}
\end{coro}
\begin{proof}
In $N(P+Q-1)^\perp$, $P'=P_V=\frac12\{1+A'+V(1-A'^2)^{1/2}\}$, thus
$$
e^{iX_{P',Q'}}=V(2P'-1)=V\{A'+V(1-A'^2)^{1/2}\}=VA'+(1-A'^2)^{1/2}.
$$
\end{proof}
In particular, if $PQ$ is S-decomposable, with singular values of simple multiplicity, one has the following
\begin{teo}
Let $PQ$ be S-decomposable, $P'Q'=\sum_{n\ge 1} s_n \langle\ \ \ , \xi_n\rangle\psi_n$, with $s_n$ of multiplicity $1$. Then $X_{P',Q'}$ is diagonalized as follows
$$
X_{P'Q'}=\sum_{n\ge 1} i \log( s_n + i (1-s_n^2)^{1/2}) \eta_n\otimes\eta_n + i \log  (s_n-i(1-s_n^2)^{1/2}))\zeta_n\otimes \zeta_n,
$$
where 
$$
\eta_n=\frac1{\sqrt{2}} \nu_n-\frac{i}{\sqrt{2}}\omega_n\  \hbox{ and } \ \ \zeta_n=\frac1{\sqrt{2}} \nu_n+\frac{i}{\sqrt{2}}\omega_n,
$$
and (as in the proof of  Theorem \ref{A diagonal}) 
$$
\nu_n=((1-s_n^2)^{1/2}-1)\xi_n +s_n\psi_n\ \hbox{ and } \ \ \omega_n=(-(1-s_n^2)^{1/2}-1)\xi_n+s_n\psi_n
$$
\end{teo}
\begin{proof}
If $PQ$ is S-decomposable, considering the decomposition of $PQ|_{N(P+Q-1)^\perp}=P'Q'$, in the proof of Theorem \ref{A diagonal}, 
$$
A'=\sum_{n\ge 1} (1-s^2_n)\nu_n\otimes\nu_n- (1-s_n^2)^{1/2}\omega_n\otimes\omega_n,
$$
for  $\nu_n, \omega_n$ described above. Then
$$
(1-A'^2)^{1/2}=\sum_{n\ge 1} s_n\nu_n\otimes\nu_n+ s_n\omega_n\otimes\omega_n.
$$
Recall that $VA=-VA$, or equivalently, $VAV=-A$ (see remarks before Theorem \ref{1.4}). Note that in $N(P+Q-1)^\perp$ we have erased the eigenvalues $\pm1$ from $A$. Then,  using Theorem \ref{A diagonal}, the fact that  the singular values of $P'Q'$ have simple multiplicity  implies that the (non nil) eigenvalues of $A'$ have single multiplicity. These two assertions imply that 
$$
V\nu_n\otimes V\nu_n=V(\nu_n\otimes\nu_n)V=\omega_n\otimes\omega_n.
$$
Thus, in the diagonalization of $A'$,  we may   replace $\xi_n, \psi_n$ by scalar multiples (of modulus one) in order that 
$$
V\nu_n=\omega_n \ \hbox{ and } \ \ V\omega_n=\nu_n.
$$
Then
$$
VA'=\sum_{n\ge 1} (1-s_n^2)^{1/2} \omega_n\otimes\nu_n - (1-s_n^2)^{1/2}\nu_n\otimes\omega_n.
$$
Thus, by the formula in the above Corollary, 
$$
e^{iX_{P',Q'}}=VA'+(1-A'^2)^{1/2}=\sum_{n\ge 1} (1-s_n^2)^{1/2} \omega_n\otimes\nu_n - (1-s_n^2)^{1/2}\nu_n\otimes\omega_n+ s_n\nu_n\otimes\nu_n + s_n\omega_n\otimes\omega_n.
$$
Note that this is a block diagonal operator, with $2\times 2$ blocks, given by the subspaces generated by the (orthonormal) vectors $\nu_n$ and $\omega_n$  for each $n$. Each block, in this basis, is given by
$$
\left( \begin{array}{cc} s_n & -(1-s_n^2)^{1/2} \\ (1-s_n^2)^{1/2} & s_n \end{array} \right),
$$
whose eigenvalues are $s_n+i(1-s_n^2)^{1/2}$ and $s_n-i(1-s_n^2)^{1/2}$, with (orthonormal) eigenvectors
$$
\eta_n=\frac1{\sqrt{2}} \nu_n-\frac{i}{\sqrt{2}}\omega_n\  \hbox{ and } \ \ \zeta_n=\frac1{\sqrt{2}} \nu_n+\frac{i}{\sqrt{2}}\omega_n,
$$
respectively, and the proof follows.
\end{proof}

Note that since $0<s_n$,  the logarithms of these eigenvalues have modulus smaller than  $\pi/2$, a fact predicted by the condition $\|X_{P',Q'}\|\le\pi/2$.

\begin{ejems}
Let us review the examples in Section 3:
\noindent

\begin{enumerate}

\item
For $I,J\subset \mathbb{R}^n$ of finite Lebesgue measure, it is known (see \cite{lenard}, \cite{folland}) that 
$$
N(P_I+Q_J-1)=\{0\}.
$$
Thus $P_I'=P_I$ and $Q_J'=Q_J$. It  is also  known (see for instance \cite{bell labs})  that in the particular case when $I$ and $J$ are intervals, the singular values of of $P_IQ_J$ have multiplicity one. Moreover the functions $\psi_n$ and $\xi_n$ are known  to be the prolate spheroidal functions, for precise $I$ and $J$ (intervals in $\mathbb{R}$) \cite{bell labs}.
It follows that one can compute explicitely the eigenvectors of $X_{P_I,Q_J}$ for such intervals $I, J$.
\item
As in Example \ref{ejemplo2},  consider $\h=L^2(\mathbb{T})$ and
$$
P=P_{\varphi\h_+} \ , \ \  Q=P_{\psi\h_+},
$$
for $\varphi, \psi$ continuous functions in $\mathbb{T}$, of modulus $1$. It was shown in \cite{grassh2} that if $\varphi$ and $\psi$ have the same winding number, then 
$$
N(P+Q-1)=\varphi\h_+\cap (\psi\h_+)^\perp \oplus (\varphi\h_+)^\perp \cap \psi\h_+=\{0\}.
$$
\item
As in example \ref{E}, let $\h=\l\times\s$ and $B:\s\to \l$ a bounded operator, $P=P_{R(E)}=P_\l$ and $Q=P_{N(E)}$ and $T=PQ$. Elementary computations show that
$$
N(P+Q-1)=R(B)^\perp\times \{0\} \oplus \{0\}\times N(B).
$$
Thus this nullspace is trivial if and only if $B$ has trivial nullspace and dense range.  Suppose that this is the case. 
Also it is straightforward to verify that 
$$
P+Q-1=\left( \begin{array}{cc}  BB^*(1+BB^*)^{-1} & -B(1+B^*B)^{-1} \\ -B^*(1+BB^*)^{-1} & -B^*B(1+B^*B)^{-1} \end{array} \right). 
$$
and that
$$
(P+Q-1)^2= \left( \begin{array}{cc}  BB^*(1+BB^*)^{-1} & 0 \\ 0 & B^*B(1+B^*B)^{-1} \end{array} \right).
$$
Then 
$$
|P+Q-1|=\left( \begin{array}{cc}  (BB^*)^{1/2}(1+BB^*)^{-1/2} & 0 \\ 0 & (B^*B)^{1/2}(1+B^*B)^{-1/2} \end{array} \right).
$$
Thus
$$
V=(P+Q-1)|P+Q-1|^{-1}=\left( \begin{array}{cc} |B^*|(1+|B^*|^2)^{-1/2}   & -B|B|^{-1}(1+|B|^2)^{-1/2} \\ -B^*|B^*|^{-1}(1+|B^*|^2)^{-1/2} & -|B|(1+|B|^2)^{-1/2} \end{array} \right).
$$
This computation is apparent if $B$ (and thus $|P+Q-1|$) is invertible, but also makes sense when $B$ has trivial nullspace and dense range. If $B=W|B|=|B^*|W$ are the polar decompositions of $B$, one has
$$
V=\left( \begin{array}{cc} |B^*|(1+|B^*|^2)^{-1/2}   & -W(1+|B|^2)^{-1/2} \\ -W(1+|B^*|^2)^{-1/2} & -|B|(1+|B|^2)^{1/2} \end{array} \right)
$$
 where $W(1+|B^*|^2)^{-1/2}$ can be replaced by $(1+|B|^2)^{-1/2}W^*$.

Therefore 
$$
e^{iX_{P,Q}}= V(2P-1)=\left( \begin{array}{cc} |B^*|(1+|B^*|^2)^{-1/2}   & W(1+|B|^2)^{-1/2} \\ -W(1+|B^*|^2)^{-1/2} & |B|(1+|B|^2)^{1/2} \end{array} \right)
$$
Suppose now that $B$ is S-decomposable, $B=\sum_{n\ge 1} s_n \langle\ \ , e_n\rangle f_n$, where since $B$ has trivial nullspace and dense range, where $\{e_n\}$ and $\{f_n\}$ are orthonormal bases of $\s$ and $\l$, respectively. Then
$$
|B|=\sum_{n\ge 1} s_n e_n\otimes e_n  \  , \ \  |B^*|=\sum_{n\ge 1} s_n f_n\otimes f_n \ ,
$$
and $W$ is a unitary operators ($W:\s\to\l$), with $We_n=f_n$. Let
$\xi_n=(e_n,0)$, $\psi_n=(0,f_n)$. Then  $\{\xi_n,\psi_n\}$ span a reducing subspace of  $T=PQ$, $P=P_{R(E)}$, $Q=P_{N(E)}$, and in view of the above formulas, also reducing for $V$ and $X_{P,Q}$. Elementary computations show that the matrix of $e^{iX_{P,Q}}$  in the basis of this reducing subspace is 
$$
\frac{1}{(1+s_n^2)^{1/2}} \left( \begin{array}{cc} s_n & 1 \\ -1 & s_n \end{array} \right).
$$
Let $\theta_n$ be defined by $\cos(\theta_n)=\frac{s_n}{(1+s_n^2)^{1/2}}$ and $\sin(\theta_n)=\frac{1}{(1+s_n^2)^{1/2}}$ (or equivalently, since $s_n>0$: $\tan(\theta_n)=\frac1{s_n}$), then the matrix of $X_{P,Q}$ in this reducing subspace is 
$$
\left( \begin{array}{cc} 0 & -i\theta_n \\ i\theta_n & 0 \end{array} \right).
$$
Recall \cite{pemenoscu} that if $P$ and $Q$ are projections such that $N(P+Q-1)=\{0\}$, there exists a unique exponent $X_{P,Q}$ with 
$d(P,Q)=\|X_{P,Q}\|$.
In particular, one has the following consequence:
\begin{coro}\label{D1}
Let $B:\s\to \l$ with trivial nullspace and dense range, and $E$ as in Example \ref{E}.
\begin{enumerate}
\item
If $B$ is invertible, then the geodesic dictance between $P_{R(E)}$ and $P_{N(E)}$ is 
$$
d(P_{R(E)},P_{N(E)})=\arctan(\|B^{-1}\|)<\pi/2.
$$
\item
If $B$ is non invertible (i.e. $B^{-1}$ is unbounded), then
$$
 d(P_{R(E)},P_{N(E)})=\pi/2.
$$
\end{enumerate}
\end{coro}
\begin{proof}
Suppose that $B$ is S-decomposable. If $B$ is invertible, $s_n\in (\|B^{-1}\|^{-1}, \|B\|)$, and if $B$ is non invertible there exists a decreasing subsequence $s_{n_k}$ of singular values of $B$, such that $s_{n_k}\to 0$. Thus the claims follow from the previous computations. 

Suppose now $B$ arbitrary. Clearly $|B|$ can be approximated by positive invertible operators $A_k$ with finite spectrum, in particular, diagonalizable. If $B=W|B|$, then $B_k=WA_k$ approximate $B$ (as in 6.10.3). Since $B$ has trivial nullspace and dense range, $W$ is a unitary operator. Then $B_k$ are S-decomposable,  with finite singular values (increasingly ordered)$s_{k,i}$, $1\le i\le n_k$.  Note that $P=P_{R(E)}$ and $Q=P_{N(E)}$ are continuous functions of $B$. Denote by  $E_k$, $P_k=P_{R(E_k)}$ and $Q_k=P_{N(E_k)}$ the operators acting in $\l\times\s$ which correspond to $B_k$. Then
$$
d(P_k,Q_k)\to d(P,Q).
$$
From the previous case, $d(P_k,Q_k)=\tan^{-1}(\frac1{s_{k,1}})$. If $B$ is invertible, $\frac1{s_{k,1}}\to \|B^{-1}\|$. Otherwise, $\frac1{s_{k,1}}\to\infty$.
\end{proof}
\end{enumerate}
\end{ejems}

\begin{rem}
As mentioned in the beginning of Section 2,  if $T=PQ$,  there may exist many factorizations, and that there exist a {\it canonical} factorization
$$
T=P_{\overline{R(T)}}P_{N(T)^\perp}
$$
with the following minimality property: for any $\xi\in\h$, and any other factorization $T=PQ$, one has
$$
\|P_{\overline{R(T)}}\xi-P_{N(T)^\perp}\xi\|\le \|P\xi-Q\xi\|.
$$. 
In \cite{incertidumbre} it was shown that in example \ref{ejemplo1} the factorization $T=P_IQ_J$ is canonical. 

In example \ref{E} suppose that $B:\s\to\l$ has trivial nullspace and dense range. Elementary computations show that for $T=P_{R(E)}P_{N(E)}$,
$$
N(T)=R(E^*) \hbox{ and } N(T^*)= N(B^*)\times\s=\{0\}\times\s.
$$
Then $\overline{R(T)}=R(E)$ and $N(T)=N(E))$, and this decomposition is also canonical. 

Also in \cite{incertidumbre}, it was shown that $R(P_I)+R(Q_J)$ is a closed proper direct sum, therefore $P_IQ_J$ is a different example from $P_{R(E)}P_{N(E)}$, for which $R(E)+N(E)$ is the whole space.
\end{rem}

\section{Dilations of contractions}

Let $\Gamma$ be a contraction in a Hilbert space $\h_0$. P.R. Halmos showed in \cite{halmos brasil}, that $\Gamma$ is the $1,1$ corner of a unitary operator $U$ acting in $\h_0\times\h_0$, namely
$$
U=\left( \begin{array}{cc} \Gamma & (1-\Gamma\Gamma^*)^{1/2} \\ (1-\Gamma^*\Gamma)^{1/2} & -\Gamma^* \end{array} \right).
$$
If 
$$
P=P_\Gamma=\left( \begin{array}{cc} 1 & 0 \\ 0 & 0 \end{array} \right)  \ \hbox{ and } \ Q=Q_\Gamma=U^*PU= \left( \begin{array}{cc} \Gamma^*\Gamma  & \Gamma^*(1-\Gamma\Gamma^*)^{1/2} \\ (1-\Gamma\Gamma^*)^{1/2}\Gamma & 1-\Gamma\Gamma^* \end{array} \right) 
$$
then
$$
\left( \begin{array}{cc} \Gamma & 0 \\ 0 & 0 \end{array} \right) =UQ_\Gamma P_\Gamma,
$$
i.e. $\Gamma$ factors as a unitary operator times a product of projections, on a bigger space.
Apparently, $\Gamma$ is S-decomposable in $\h$ if and only if $QP$ is decomposable in $\h\times\h$

Moreover, if
$$
\Gamma=\sum_{n\ge 1} s_n \langle\ \ \ , \xi_n\rangle\psi_n,
$$
then
$$
QP=\sum_{n\ge 1} s_n \langle\ \ \ \ , \left( \begin{array}{c} \xi_n \\ 0 \end{array} \right)\rangle \left( \begin{array}{c} s_n\xi_n \\ (1-s_n^2)^{1/2} \psi_n \end{array} \right).
$$

\begin{lem}
On $\h\times\h$, one has that 
$$
R(P)\cap R(Q)=N(1-\Gamma^*\Gamma)\oplus 0 , \ N(P)\cap N(Q)=\{0\}\oplus N(1-\Gamma\Gamma^*)
$$
and
$$
R(P)\cap N(Q)=N(\Gamma)\oplus \{0\} , \ N(P)\cap R(Q)=\{0\}\oplus N(\Gamma^*).
$$
\end{lem}
\begin{proof}
A vector in $R(P)$ is of the form $\xi=\left( \begin{array}{c} \xi_1 \\ 0 \end{array} \right)$.  $Q\xi=\xi$ if and only if 
$$
\Gamma^*\Gamma\xi_1=\xi_1 \hbox{ and } (1-\Gamma\Gamma^*)^{1/2}\Gamma\xi_1=0.
$$
Note that $(1-\Gamma\Gamma^*)^{1/2}\Gamma=\Gamma(1-\Gamma^*\Gamma)^{1/2}$. Thus $\Gamma^*\Gamma\xi_1=\xi_1$  implies that $(1-\Gamma^*\Gamma)^{1/2}\xi_1=0$.

A vector $\xi\in R(P)$ belongs to $N(Q)$ if and only if
$\Gamma^*\Gamma\xi_1=0$, i.e., $\xi_1\in N(\Gamma)$.

The other two statements are similar.
\end{proof}
\begin{rem}
Straightforward computations show that 
$$
(P+Q-1)^2=\left( \begin{array}{cc} \Gamma^*\Gamma & 0 \\ 0 & \Gamma\Gamma^* \end{array} \right) \hbox{ and thus } |P+Q-1|=\left( \begin{array}{cc} (\Gamma^*\Gamma)^{1/2} & 0 \\ 0 & (\Gamma\Gamma^*)^{1/2} \end{array} \right) .
$$
Suppose that $N(\Gamma)=N(\Gamma^*)=\{0\}$ (i.e., $P+Q-1$ has trivial nullspace and dense range).
If $\Gamma=W|\Gamma|=|\Gamma^*|W$ are the polar decompositions (with $W$ a unitary operator), then
$$
V=\left( \begin{array}{cc} |\Gamma| & W^*(1-\Gamma\Gamma^*)^{1/2} \\ (1-\Gamma\Gamma^*)^{1/2}W & -|\Gamma^*| \end{array} \right)
\hbox{ and } e^{iX_{P,Q}}=\left( \begin{array}{cc} |\Gamma| & -W^*(1-\Gamma\Gamma^*)^{1/2} \\ (1-\Gamma\Gamma^*)^{1/2}W & |\Gamma^*| \end{array} \right)
$$
With similar computations as above, one sees that if $\Gamma$ is S-decomposable with singular values $0<s_n\le 1$, then 
the spectrum of $X_{P,Q}$ is $\{\pm\theta_n: cos(\theta_n)=s_n\}$. 
With an  argument as in Corollary \ref{D1}, one has:
\begin{coro} Let $\Gamma$ be a contraction in $\h_0$ with trivial nullspace and dense range, and $P_\Gamma$, $Q_\Gamma$ the above projections in $\h_0\times\h_0$.  
\begin{enumerate}
\item
If $\Gamma$ is invertible, then 
$$
d(P_\Gamma,Q_\Gamma)=\cos^{-1}(\|\Gamma^{-1}\|^{-1}).
$$
\item
If $\Gamma$ is non invertible, then
$$
d(P_\Gamma,Q_\Gamma)=\pi/2.
$$
\end{enumerate}
\end{coro}

\end{rem}

{\sc (Esteban Andruchow)} {Instituto de Ciencias,  Universidad Nacional de Gral. Sar\-miento,
J.M. Gutierrez 1150,  (1613) Los Polvorines, Argentina and Instituto Argentino de Matem\'atica, `Alberto P. Calder\'on', CONICET, Saavedra 15 3er. piso,
(1083) Buenos Aires, Argentina.}

e-mail: eandruch@ungs.edu.ar

{\sc (Gustavo Corach)} {Instituto Argentino de Matem\'atica, `Alberto P. Calder\'on', CONICET, Saavedra 15 3er. piso, (1083) Buenos Aires, Argentina, and Depto. de Matem\'atica, Facultad de Ingenier\'\i a, Universidad de Buenos Aires, Argentina.}

e-mail: gcorach@fi.uba.ar
\end{document}